\documentclass[10pt,twoside]{amsart}

\usepackage{amsmath,amssymb,bbm}
\usepackage{amsfonts,amssymb}
\usepackage{color}

\usepackage{mathtools}
\usepackage{enumerate}

\usepackage[colorlinks,linkcolor=blue,citecolor=magenta,hypertexnames=false]{hyperref}

\theoremstyle{plain} %% This is the default

\numberwithin{equation}{section}

\newtheorem{thm}{Theorem}[section]
\newtheorem{cor}[thm]{Corollary}
\newtheorem{lem}[thm]{Lemma}
\newtheorem{prop}[thm]{Proposition}

\newtheorem{defn}[thm]{Definition}

\theoremstyle{remark}
\newtheorem{rem}[thm]{Remark}
\newtheorem*{ack}{Acknowledgment}

\def\la{{\langle}}
\def\ra{{\rangle}}
\def\f{\frac}

\def\({\left(}
\def \){ \right)}
\def\[{\left[}
\def \]{ \right]}
\def\Bl{\Bigl}
\def\Br{\Bigr}
\def\DL{D_{L^2, \text{cap}}}
 \def\Ga{\Gamma}
\def\ta{\theta}
\def\al{{\alpha}}
\def\be{{\beta}}
\def\da{{\delta}}
\def\sa{{\sigma}}

 \def\t{{\theta}}
 \def\l{{\lambda}}

 \def\s{{\sigma}}
 
 \def\va{\varepsilon}
 \def\NN{{\mathbb N}}

 \def\RR{{\mathbb R}}
 \def\SS{{\mathbb S}}

  \def\sph{\mathbb{S}^{d}}

\newcommand{\wh}{\widehat}

\usepackage{mathtools}
\mathtoolsset{showonlyrefs,showmanualtags}

\begin{document}

\def\diam{\text{diam}}
\def\p{\partial}
\def\ld{\lambda}
\def\bl{\bigl}
\def\br{\bigr}
\def\og{\omega}
\def\Ld{\Lambda}
%\usepackage{amsmath,amssymb}
%% my poor man's solution to arc notation
%\newcommand{\tarc}{\mbox{\large$\frown$}}
%\newcommand{\arc}[1]{\stackrel{\tarc}{#1}}
%%%
%\pagestyle{empty}

%\title[]{}\maketitle

\title{Geodesic distance Riesz energy on the sphere}

\author{Dmitriy Bilyk}
\address{School of Mathematics, University of Minnesota, Minneapolis, MN 55408, USA.}
\email{dbilyk@math.umn.edu}

\author{Feng Dai}
\address{Department of Mathematical and Statistical Sciences\\
University of Alberta\\ Edmonton, Alberta T6G 2G1, Canada.}
\email{fdai@ualberta.ca}

\begin{abstract}
We study energy integrals and discrete energies on the sphere, in particular, analogs of the  Riesz energy with  the geodesic distance in place of Euclidean, and observe that the range of exponents for which the uniform distribution optimizes such energies is different from the classical case. We also obtain a general form of the Stolarsky principle, which relates discrete energies to certain $L^2$ discrepancies. This leads to new proofs of discrepancy estimates, as well as the sharp asymptotics of the difference between optimal discrete and continuous energies in the geodesic case.
\end{abstract}

\maketitle

\section{Introduction}

In the present paper we study optimization properties of some  energy integrals and discrete energies on the unit  sphere $\mathbb S^d \subset \mathbb R^{d+1}$, in particular those  related to the geodesic distance on the sphere. \\ %The  general formulation of the problem is as follows.

Let  $\mathcal{B}$   denote the collection of all Borel probability measures on $\sph$. %Let   $\da_{x_0} $ the Dirac Borel  probability measure supported at $x_0\in\sph$.
 Given a measure  $\mu\in\mathcal{B}$,
 define the energy integral $I_F (\mu )$ of a measurable (non-negative or bounded) function $F:[-1,1]\to\RR$
 by
\begin{equation}\label{EI}
I_F (\mu  ):=\int_{\sph}\int_{\sph} F(x\cdot y)\, d\mu(x) \, d\mu(y).
\end{equation}
We are  interested in finding the optimal (maximal or minimal, depending on $F$) values of $I_F (\mu  )$ over $\mu \in \mathcal B$, as well as extremal  measures $\mu$ for which this values are achieved, i.e. equilibrium distributions with respect to $F$.  In particular, it is natural to investigate whether the Lebesgue surface measure $\sigma$ normalized by $\sigma (\sph) =1$ is a minimizer (maximizer), and if so, whether it is unique --  in other words, whether  optimizing the  energy with potential $F$ induces   uniform distribution.\\

For a finite set of $N$  points $Z= \{ z_1, \dots, z_N\}$ in $\sph$, its discrete energy with respect to $F$ is defined as
\begin{equation}\label{DE}
 E_F (Z  ) = \sum_{1\le i <j \le N} F(z_i \cdot z_j) .
 \end{equation}
  Note that in the case when $F(1) = 0$, we have $$E_F (Z)  = \frac12 N^2 I_F (\mu  ) \,\,\, \textup{ with } \,\,\, \mu = \f1N\sum_{i=1}^N \delta_{z_i},$$ where $\delta_x$ is the Dirac mass at $x \in \mathbb S^d$.  Appropriate modifications accounting for the diagonal terms should be made in the general case. One is interested in optimizing  the discrete energy for a given  $N$, analyzing extremal $N$-point configurations, comparing optimal values of the discrete energy to the optimal energy integral, and finding asymptotic behavior of this difference.

Such problems  arise naturally in various fields, e.g., in  electrostatics, in the study of  equilibrium distributions of charges  which repel according to the law given by $F$. One of the most natural choices of the potential   is the so-called Riesz potential $F(x\cdot y) = | x-y |^{-s}$, where $|x-y|$ is the Euclidean distance between $x$ and $y$ in $\mathbb R^{d+1}$. In particular, for $d=2$ and $s=1$, minimizing the energy \eqref{DE} amounts to finding the equilibrium (according to Coulomb's Law) distributions of $N$ electrons on the sphere. This situation has been studied extensively and numerous questions posed above are well understood in this case \cite{Bjorck,BHS,Brauchart,KS,schoen,wagner1,wagner2}, although  precise optimal discrete distributions are still elusive for most values of $N$.

While our present work establishes many general facts and relations,  we primarily  concentrate on the case which is seemingly similar to the classical  Riesz energy, but uses geodesic, rather than Euclidean, distance in the definition of energy.  This object naturally arose   in the companion paper of the authors \cite{BD1} in relation to discrepancy theory and Stolarsky principle, and has also been considered previously \cite{F-T,Sperling, Larcher,BHS}. To make things precise, let $\rho (x,y)$ denote the geodesic distance between $x$ and $y$ on $\mathbb S^d$, %normalized so that  the distance between antipodal poles is $1$,
i.e.
\begin{equation}
\rho (x,y) = %\frac{1}{\pi}
\arccos ( x\cdot y ).
\end{equation}
We shall consider energies defined by the function
\begin{equation}
F_\delta (x\cdot y) = \big( \rho (x,y) \big)^\delta, \,\,\, \textup{ i.e. } \,\,\, F_\delta ( t ) = (\arccos t)^\delta,
\end{equation}
for an arbitrary  $\delta \in \mathbb R\setminus \{0\}$; for $\delta = 0$, the standard modification is the logarithmic potential $ F_0 (t) =   - \log \big( \frac{1}{\pi} \arccos t \big)$.  We would like to characterize extremizers of the energy integral
\begin{equation}\label{gdei}
I_{d,\delta} (\mu ) =  I_{F_\delta} (\mu) %= \int_{\sph}\int_{\sph} F_\da(x\cdot y)\, d\mu(x)d\mu(y)
= \int_{\sph}\int_{\sph}   \big( \rho(x,y) \big)^\delta   \, d\mu(x)d\mu(y),
\end{equation}
which we shall refer to as the {\emph{geodesic distance (Riesz) energy}}. Naturally one is interested in minimizers when $\delta \le 0$ and maximizers for $\delta >0$.

One may expect that the behavior  of the geodesic distance energy should be similar to its Euclidean counterpart, i.e. the standard Riesz energy.  Perhaps surprisingly, this is not quite the case. In dimension $d=1$ (on the circle) this phenomenon has been previously  observed  in \cite{BHSd}: in the geodesic case the uniform distribution $\sigma$ ceases to be the unique extremizer of $I_{d,\delta}$ when $\delta \ge 1$, while in the case of Riesz energy the analogous critical value is $\delta = 2$.

In the present work together with  our companion paper \cite{BD1} we prove this  fact in all dimensions $d\ge 1$. More precisely, we prove the following theorem:
\begin{thm}\label{GeodExt}
Let $I_{d,\delta}   (\mu)$ be the geodesic distance energy integral on $\sph$, with $\delta \in \mathbb R$, as defined in \eqref{gdei}.
%\begin{equation}
%I (\mu) = \int\limits_{\mathbb S^d} \int\limits_{\mathbb S^d} \big( d(x,y) \big)^\delta  \, d\mu (x) d\mu (y)
%\end{equation}
The extremizers  of this energy integral over  $\mu\in \mathcal B$  can be characterized as follows:
\begin{enumerate}[(i)]
\item\label{G1} $-d < \delta \le 0$: the unique minimizer of $I_{d,\delta} (\mu)$ is $\mu = \sigma$ (the normalized surface measure).
\item\label{G2} $0 < \delta < 1$: the unique maximizer of $I_{d,\delta} (\mu)$ is $\mu = \sigma$.
\item\label{G3} $\delta = 1$:  $I_{d,\delta} (\mu)$ is maximized if and only if  $\mu$ is centrally symmetric .
\item\label{G4} $\delta > 1$: $I_{d,\delta} (\mu)$ is maximized if and only if  $\mu = \frac12 (\delta_p + \delta_{-p})$, i.e. the mass is equally  concentrated at two antipodal poles.
\end{enumerate}
\end{thm}
The last two parts of this theorem are proved by the authors and R. Matzke  in \cite{BD1}: the critical case \eqref{G3} is obtained as a consequence of the geodesic distance Stolarsky principle, relating hemisphere discrepancy and the geodesic distance energy (see also \S \ref{s.stol} in the present paper), while the degenerate case \eqref{G4} easily follows from the critical case. In the present paper  we present the proof of the first two parts, \eqref{G1} and \eqref{G2}, in which optimal   energy leads to uniform distribution, see Theorem \ref{thm-2-1}.

In the one-dimensional case   parts \eqref{G1} and \eqref{G2} of Theorem \ref{GeodExt} have been previously established by Brauchart, Hardin, and Saff \cite{BHSd}, along with precise asymptotic energy of the discrete geodesic energy of $N$ equally spaced points. Immediately after our result, Tan \cite{Tan} gave an alternative proof of parts \eqref{G2}-\eqref{G4} of Theorem \ref{GeodExt}, see the remark at the very end of \S \ref{s.gdei}.

We should note that in the case of Riesz energy, i.e. for $ F(x\cdot y) = | x-y |^\delta$, part \eqref{G1} also holds, see e.g. \cite{KS}, while for $\delta >0$ the situation is somewhat different, as was established by Bjorck \cite{Bjorck}:   $\sigma$  is the unique maximizer when $\delta \in (0,2)$, while for $\delta >2$  maximizers collapse to symmetric two-point measures as in \eqref{G4}; at the critical value $\delta =2$ maximizers are precisely those measures whose center of mass is at the origin. Intuitively, for small values of $\delta$, in particular, when $\delta <0$,  small scale interactions contribute the most to the energy, therefore (since $\rho(x,y) \approx |x-y|$ when $x$ and $y$ are close) both energies exhibit similar behavior, while for larger values of $\delta$ mid-range and long-range interactions come into play and the difference between geodesic and Euclidean distances manifests itself in the energy integrals. The restriction that $\delta > -d$ is natural, since on a $d$-dimensional manifold, the corresponding energy integral with $\delta \le -d$ would be infinite for any $\mu \in \mathcal B$.

Our proofs rely  on {\emph{spherical harmonic}} expansions. We briefly review the basic notions in \S \ref{prelim}, but for a detailed and extensive exposition the reader is directed to, e.g., \cite{DX}. In \S\ref{prelim} we discuss  connections between the extremizers  of   energy integrals and  properties of the potential $F$ (signs of  the Gegenbauer coefficients, positive definiteness). Some of these connections are well known in the theory and go back to Schoenberg \cite{schoen}, while some formulations are new.  We would like to point out  that some of these properties are also discussed in our our parallel paper \cite{BD1} without resorting to (or with minimal use of) spherical harmonics. \\

  In \S \ref{s.gdei} we apply the general results   presented  in \S \ref{prelim} to the specific case of the geodesic distance energy integral \eqref{gdei} and prove parts \eqref{G1} and \eqref{G2} of Theorem \ref{GeodExt}: these results are contained in Theorem \ref{thm-2-1}. Up to some technical details, the proofs boil down to demonstrating that the Gegenbauer coefficients of the potential  are all positive (negative). Essential computations are carried out in Lemma \ref{lem-2-1}. \\

  In \S \ref{s.stol} we connect to different objects, which quantify equidistribution:   we show (part \eqref{St1} of Theorem \ref{thm-3-2}, see also \S 5 in \cite{BD1} for a more detailed discussion) that for all positive definite functions $F$ the difference between the discrete and continuous energies may be represented as the $L^2$ norm of a certain discrepancy function: this is a generalization of the {\it{Stolarsky principle}} in discrepancy theory \cite{stol}. Furthermore, (see part \eqref{St2} of Theorem \ref{thm-3-2}) this discrepancy may be estimated using the function $F$, in particular, lower bounds involve Gegenbauer coefficients of $F$. We apply these results to give an alternative proof of a classical bound on the spherical cap discrepancy due to Beck \cite{Beck}, see Theorem \ref{beck}. \\

Next, in  \S \ref{s.disc}, we turn to the problem of estimating the asymptotic difference between the geodesic distance energy of the uniform distribution $I_{d,\delta} (\sigma)$ and the corresponding optimal energy of discrete $N$-point distributions, as $N\rightarrow \infty$. Setting $\displaystyle{\mathcal{E}_{d,\da} (N) = \inf_{\# Z  = N} E_{F_\delta}  (Z)}$, where as before    $ F_\delta ( t ) = (\arccos t)^\delta$, and     using the results of \S\ref{s.stol}, in Theorem \ref{thm-1-4-energy-ch14} we establish that  the asymptotic estimate
\begin{equation}
 I_{d,\da} (\sa) - \frac{2}{N^2} \mathcal{E}_{d,\da}  (  N)\sim  N^{-1-\f \da d}\,
 \end{equation}
 holds  for $-d < \delta < 1$ (with a logarithmic correction for $\delta =0$). This closely mirrors  the case of the classical Riesz energy, but for  the Euclidean distance this estimate   is valid for $-d < \delta <2$ (this has been established in a series of papers: \cite{wagner1,wagner2}, \cite{KS}, and   \cite{Brauchart}). 

% In order to obtain this result  in \S \ref{s.stol}
In order to prove Theorem \ref{thm-1-4-energy-ch14}, one needs  sharp asymptotic estimates of the Gegenbauer coefficients of $F_\delta$ (while  to establish the optimality of $I_{d,\delta} (\sigma)$ in \S \ref{s.gdei}, it suffices  just to show  that these coefficients are positive). These bounds, which in the geodesic case are much more complicated than for the Euclidean distance,    are stated in Lemma \ref{lem-4-3} and their rather technical proof is presented in \S\ref{s.tech}. %The same method can be used to prove sharp asymptotics for the discrete Riesz energy, which was previously done in \cite{wagner1,wagner2,KS,Brauchart}. In addition,  we use the approach of \S \ref{s.stol} to give an alternative proof of Beck's  classical lower bound of the spherical cap discrepancy of an $N$ point set, see Theorem \ref{beck}.

\begin{ack}
The authors are  extremely grateful to CRM (Barcelona): their collaboration has started while both of them participated in the research program on  ``Constructive Approximation and Harmonic Analysis" in  2016. The stay of the first author at CRM has been sponsored by NSF grant DMS 1613790. This work is partially supported by NSERC  Canada under
grant  RGPIN 04702 (Dai) and  by the Simons foundation collaboration grant (Bilyk).   \end{ack}

\section{Preliminaries}\label{prelim}

Let $w_\ld(t)=(1-t^2)^{\ld-\f12}$ with $\ld>0$.  Given $1\leq p<\infty$, we  denote by $L_{w_\ld}^p[-1,1]$ the space of all real  integrable functions $F$ on $[-1,1]$ with $\|F \|_{p,\ld}:=\Bl(\int_{-1}^1 |F(t)|^p w_\ld(t)\, dt\Br)^{1/p}<\infty$.
Every function $F\in L_{w_\ld}^1[-1,1]$ has  a Gegenbauer (ultraspherical)  polynomial expansion:
 \begin{equation}\label{1-1-16}F(t)\sim \sum_{n=0}^\infty \wh{F}(n; \ld) \f {n+\ld} \ld C_n^\ld(t),\  \  t\in [-1,1],\end{equation}
where $C_n^\lambda$ are Gegenbauer polynomials (see \cite{DX} for an extensive discussion) and
\begin{align*}
\wh{F}(n; \ld) &=\f{\Ga(\ld+1)}{\Ga(\ld+\f12)\Ga(\f12)}
 \int_{-1}^1  F(t) R_n^{\ld}(t) (1-t^2)^{\ld-\f12}\, dt,  \   \ n=0,1,\cdots,
\end{align*}
where $R_n^\ld(t)=\f{C_n^\ld(t)}{C_n^\ld(1)}$.  From now on in this text we shall set the value of $\lambda$  to  $$\lambda = \frac{d-1}{2}.$$ In the special case $\lambda =0$ (which corresponds to the circle $\mathbb S^1$) one obtains Chebyshev polynomials of the first kind $T_n (t)$, which satisfy $$ \frac{1}2 \lim_{\lambda \rightarrow 0} \frac{n+\lambda}{\lambda} C_n^\lambda (t) = T_n (t) = \cos \big( n \arccos t \big).$$

  Denote by $\s$  the  surface Lebesgue measure on $\sph$ normalized by $\sa(\sph)=1$, and let $\mathcal{H}_n$ be  the space of all spherical harmonics of degree $n$ on $\sph$.  Let   $\{ Y_{n,1},\cdots, Y_{n, a_n^d}\}$ denote  a real  orthonormal basis  of the space $\mathcal{H}_n$. The addition formula for spherical harmonics  states that (see, for instance, \cite[1.2.8]{DX}) % $x$, $y \in \mathbb S^d$
\begin{equation}\label{1-1-1}
    \sum_{j=1}^{a_n^d} Y_{n,j}(x) Y_{n,j}(y) = \f {n+\ld}{\ld} C_n^\ld(x\cdot y)\   \ \textup{ for all }\,\, x,y\in \sph,
\end{equation}
where $$a_n^d=\f {n+\ld}{\ld} C_n^\ld(1)=\text{dim}\ \mathcal{H}_n\sim n^{d-1},\   \  \ld=\f{d-1}2.$$

 As before we denote   the collection of all Borel probability measures on $\sph$  by
$\mathcal{B}$, and    $\da_{x_0} $  is the Dirac Borel  probability measure supported at $x_0\in\sph$.
Given $\mu\in\mathcal{B}$,
 define the energy integral $I_ F (\mu )$ of a (bounded or nonnegative) measurable function $F:[-1,1]\to\RR$ as in \eqref{EI}
 by
$ \displaystyle{I_F (\mu) =\int_{\sph}\int_{\sph} F(x\cdot y)\, d\mu(x) \, d\mu(y)}$.  We have the following proposition on extremizers of $I_F (\mu) $ over $\mathcal B$:

 \begin{prop}\label{prop-1-2} Let $\ld=\f{d-1}2$ and let  $F$ be  a  continuous function  on $[-1,1]$. The following conditions are equivalent:
 \begin{enumerate}[(a)]
\item\label{a}    $\wh{F}(n; \ld)\ge 0$  for all $n\ge 1$;
\item\label{b} the  surface Lebesgue measure  $\mu=\s$ on $\sph$ is a minimizer   of  the energy  integral $ I_F (\mu) $;
%\item\label{c}  every   Dirac mass $\mu=\da_e$, $e\in\sph$ is  a maximizer   of  the energy  integral $ I_F (\mu )$.
 \end{enumerate}
 If the above conditions hold, then every   Dirac mass $\mu=\da_e$, $e\in\sph$ is  a maximizer   of  the energy  integral $ I_F (\mu )$.
 \end{prop}

  Moreover, concerning uniqueness of extremizers, the following holds:
  \begin{prop}\label{prop-1-2u} Let $\ld=\f{d-1}2$ and let  $F\in C[-1,1]$. The following conditions are equivalent:
   \begin{enumerate}[(a)]
 \item\label{aa} $\wh{F}(n;\ld)> 0$  for all $n\ge 1$;
 \item\label{bb} the normalized surface measure $\sigma$ is the unique  minimizer of $I_F (\mu )$;
% \item\label{cc} every maximizer of $I_F (\mu )$  is a Dirac mass.
 \end{enumerate}
 In this case, every maximizer of $I_F (\mu )$  is a Dirac mass.
      \end{prop}

Obviously changing the inequality signs in the above proposition reverts the roles of maximizers and minimizers. In addition, since adding constants to $F$ does not change the extremizers of $I_F (\mu)$, it is natural that $\wh{F} (0; \lambda) $ does not play a role in these statements.  In order to prove these statements we shall need the following technical lemma:

\begin{lem}\label{l.decay}
Let $F\in C [-1,1]$ and assume that    $\wh{F}(n; \ld)\ge 0$  for all $n\ge 1$. Then  \begin{equation}\label{1-1-a}
   \sum_{n=1}^\infty  n^{2\ld}\, \wh{F}(n;\ld) <\infty,
 \end{equation}
  which, in particular, will imply that the series on the right hand side of \eqref{1-1-16} converges uniformly and absolutely to the function $F$  on $[-1,1]$.
\end{lem}

\begin{proof}
%It remains to show the claim \eqref{1-1-a}.
Let $\sa_n^{\da} F$ denote the Ces\`aro $(C, \da)$-means of the Gegenbauer polynomial expansion of $F$ (see \cite[Section 2.4]{DX} for details), i.e.
$$\sa_n^\da F(t)=\sum_{k=0}^n \f {A_{n-k}^\da}{A_n^\da} \f{k+\ld}{\ld}\wh{F}(k;\ld)C_k^\ld (t),\    \    \   \  A_j^\da=\f{\Ga(j+\da+1)}{\Ga(j+1)\Ga(\da+1)}.$$
It is known (e.g. Theorem 2.4.3 in \cite{DX})  that for $\da>\ld$,
$$\lim_{n\to\infty}\|\sa_n^\da F-F\|_{L^\infty[-1,1]}=0,\   \  \forall F \in C[-1,1].$$
On the other hand, since  for each fixed $j$, the sequence $\{ \f {A_{n-j}^\da}{A_n^\da}\}_{n=j}^\infty$ increases to $1$ as $n\to\infty$, it follows  by Levi's monotone convergence theorem  that
\begin{align*}
  f(1)=\lim_{n\to\infty} \sa_n^\da (f)(1)=\lim_{n\to\infty} \sum_{k=0}^n \f{A_{n-k}^\da}{A_n^\da } \wh{f}(k;\ld) \f {k+\ld} \ld C_k^\ld (1)=\sum_{k=0}^\infty \wh{f}(k;\ld) \f {k+\ld} \ld C_k^\ld (1).
\end{align*}
This yields \eqref{1-1-a}  since  $(k+\ld)C_k^\ld(1)\sim k^{2\ld}$ as $k\to\infty$.
\end{proof}

{\emph{Remark:}} The self-improving property \eqref{1-1-a} (positivity of coefficients implies their decay) has various manifestations  in harmonic analysis: e.g., if  a function $f\in L^1 (\mathbb T)$ has  Fourier series $\sum i  c_n e^{2\pi i nx}$ with $c_n = - c_{-n} \ge 0$ (i.e. sine series with non-negative coefficients), then necessarily $\sum_{n>0} \frac{c_n}{n} <\infty$ (see e.g. \cite[page 24]{Katz}), which is a direct analog of \eqref{1-1-a}.

We now prove  Proposition \ref{prop-1-2}:

    \begin{proof}[Proof of Proposition \ref{prop-1-2}] %We start with the first set of equivalences. Assume that \eqref{a}  holds, i.e. that  $\wh{F}(n; \ld)\ge 0$  for all $n\ge 1$. We claim that in this case
 %\begin{equation}\label{1-1-a}
   %\sum_{n=1}^\infty |\wh{F}(n;\ld)| n^{2\ld}<\infty,
 %\end{equation}
%  which, in particular, will imply that the series on the right hand side of \eqref{1-1-16} converges uniformly and absolutely to the function $F$  on $[-1,1]$.

 %  For the moment, we take   \eqref{1-1-a} for granted and proceed with the proof  of the facts  
  
 We first prove that $\s$ is a minimizer of $I_F (\mu )$ and $\da_e$ is a maximizer of $I_F (\mu)$ for any $e\in\sph$. Indeed,
by  \eqref{1-1-a},  \eqref{1-1-1} and the  dominated convergence theorem, it follows  that
\begin{align}
    I_F (\mu)   &=\sum_{n=0}^\infty \wh{F}(n;\ld) \int_{\SS^d}\int_{\SS^d} \f {n+\ld} \ld C_n^\ld (x\cdot y) \, d\mu(x)\, d\mu(y)\notag\\
     &=\sum_{n=0}^\infty \wh{F}(n;\ld) \sum_{j=1}^{a_n^d} \int_{\SS^d}\int_{\SS^d} Y_{n,j}(x) Y_{n,j}(y)\, d\mu(x) d\mu(y)\\ & =\wh{F}(0;\ld)+\sum_{n=1}^\infty \wh{F}(n;\ld) b_{n,\mu},\label{1-3-1}
\end{align}
where, by the addition formula \eqref{1-1-1},
\begin{equation}\label{1-5}
    b_{n,\mu} =\f {n+\ld} \ld\int_{\SS^d}\int_{\SS^d}  C_n^\ld (x\cdot y) \, d\mu(x)\, d\mu(y)=\sum_{j=1}^{a_n^d} \bigg( \int_{\sph} Y_{n,j}(x)\, d\mu(x)\bigg)^2 \ge 0.
\end{equation}
 Using the Cauchy--Schwarz  inequality and \eqref{1-1-1} with $x=y$,  we obtain
 \begin{align}\label{1-4-1}
    0\leq b_{n, \mu} \leq \int_{\sph}\sum_{j=1}^{a_n^d}  |Y_{n,j}(x)|^2\, d\mu(x)=
    a_n^d.
 \end{align}
 Thus, if $\wh{F}(n;\ld)\ge  0$ for all $n\ge 1$, then by \eqref{1-3-1} and \eqref{1-4-1}, one has $$
    I_F (\mu) \ge \wh{F}(0;\ld)=I_F (\sa ) \,\, \textup{ and }$$
    $$ I_F (\mu ) \leq \wh{F}(0;\ld) +\sum_{n=1}^\infty \wh{F}(n;\ld) a_n^d =F(1) =I_F ({\da_{e}} ),$$ i.e.
    $\s$ and $\da_e$ are a minimizer and a maximizer of the integral $I_F (\mu )$, respectively.\\

Conversely, if  $\wh{F}(n; \ld) < 0$ for some $n\ge 1$, then set $d\mu (x) = \big(1 + \varepsilon Y_{n,1} (x)  \big) d\sigma (x)$, where $\varepsilon >0$  is chosen small enough so that $   1 + \varepsilon Y_{n,1} (x)  \ge 0$ on $\mathbb S^d$.  The Funk--Hecke formula (see, e.g., Theorem 1.2.9 in \cite{DX}) states that for any spherical harmonic $Y \in \mathcal H_n$
 \begin{equation}\label{e.FH}
  \int_{\mathbb S^d} F (x\cdot  y) Y (x)  d\sigma (x)  = \wh{F} (n;\lambda) Y (y).
  \end{equation}
Thus, using the fact that $\int_{\sph} Y_{n,1} (x) d\sigma (x) = 0$,  we find that $\mu \in \mathcal B$ and
\begin{align}
\label{e.perturb}  I_F (\mu) & =  \int_{\sph}\int_{\sph} F(x\cdot y)\, \big(1 + \varepsilon Y_{n,1} (x)  \big) \big(1 + \varepsilon Y_{n,1} (y)  \big)\, d\sigma (x) \, d\sigma (y)\\
\nonumber & = I_F (\sigma) +  \varepsilon^2  \, \wh{F} (n;\lambda)   \int_{\sph} Y^2_{n,1} (y) d\sigma (y) < I_F (\sigma),
\end{align}
i.e. $\sigma $ is not a minimizer of $I_F$.
\end{proof}

{\emph{Remark:}} We observe that the fact that $\delta_e$ maximizes $I_F$ is not equivalent to conditions \eqref{a}-\eqref{b} of Proposition \ref{prop-1-2}. Indeed, a sufficient condition for this is that $\max_{t\in[-1,1]} F (t) =F (1)$, since then for each $\mu \in \mathcal B$, we have $I_F (\mu) \le \| F \|_\infty = F(1) = I_F (\delta_e)$. For example, for $F(t) = - \big( \arccos t \big)^2$, the maximizer is obviously $\delta_e$, while the minimizer is not $\sigma$: according to part \eqref{G4} of Theorem \ref{GeodExt}, minimizers are measures of the form $\frac12 ( \delta_e + \delta_{-e})$.

%By orthogonality, $b_{k,\mu} = 0$ when $k\neq n$, and    $$   b_{n,\mu}  = \bigg( \int_{\sph} Y_{n,1}(x)\, d\mu(x)\bigg)^2  = \bigg( \int_{\sph} Y_{n,1}(x)\,  \big(1 + \varepsilon Y_{n,1} (x)  \big) d\sigma (x) \bigg)^2  = \varepsilon^2 $$

We now prove Proposition \ref{prop-1-2u} about the uniqueness of  minimizers.
\begin{proof}[Proof of Proposition \ref{prop-1-2u}] By Proposition \ref{prop-1-2} it is enough to assume that $\wh{F}(n;\ld) \ge 0$  for all $n\ge 1$.
 Suppose that \eqref{aa} holds, i.e. $\wh{F}(n;\ld)> 0$  for all $n\ge 1$.  In this case,   by \eqref{1-3-1}, the equality
    $I_F (\mu )  = \wh{F}(0;\ld)=I_F (\sa  )$  holds  if and only if $b_{n,\mu}=0$ for all $n\ge 1$, i.e. if and only if  $\int_{\sph} g(x)\, d\mu(x)=0$ for all $g\in\mathcal{H}_n$ and $n\ge 1$. This last condition implies that for each spherical polynomial $P$ on $\sph$,
    $$\int_{\sph} P(x)\, d\mu(x)=\int_{\sph} P(x)\, d\s(x).$$
    By the density of spherical polynomials in the space $C(\sph)$, we then conclude that $d\mu=d\s$.

       Next, we show that if $\mu_0\in\mathcal{B}$ is a  maximizer of $I_F (\mu) $, then   $\mu_0=\da_e$ for some $e\in\sph$.
        To see this, we first note that    according to  \eqref{1-3-1}, \eqref{1-5} and \eqref{1-4-1},
         in order that $$I_ F({\mu_0}) =\max_{\mu\in\mathcal{B}} I_F (\mu)  =  I_F (\delta_e) = \wh{F}(0;\ld) +\sum_{n=1}^\infty \wh{F}(n;\ld) a_n^d,  $$ it is necessary that
         $$ \Bl( \int_{\sph} Y_{n,j}(x)\, d\mu_0(x)\Br)^2 =\int_{\sph} |Y_{n,j}(x)|^2 \, d\mu_0(x),\  \  \forall n\ge 1, \  \ \forall 1\leq j\leq a_n^d,$$
        or equivalently,
       \begin{equation}\label{1-6-0}
Y_{n, j}(x)\equiv constant\    \  \text{$\mu_0$-a.e. on $\sph$},\   \   \forall n\ge 1,\   \ \forall 1\leq j\leq a_n^d,
        \end{equation}
          which in turn implies that each spherical polynomial is constant $\mu_0$-a.e. on $\sph$.
          Since the space of spherical polynomials is dense in $C(\sph)$, we further conclude  that every continuous function on $\sph$ is constant $\mu_0$-a.e. on $\sph$.

 Assume that the support of $\mu_0$ contains at least two distinct points $e_0 \neq e_1$. (Recall that $\operatorname{supp} \mu_0$ is the complement of the union of all open sets $U$ with $\mu_0 (U) =0$.) Choose open neighborhoods $U_i$ of $e_i$ such that $\mu_0 (U_i) >0$ and $U_0 \cap U_1 = \emptyset$, and construct a function $f\in C (\sph)$ with $f \vert_{U_i}   = i $. Then $f$ is not constant $\mu_0$-a.e. on $\sph$. Thus $\operatorname{supp} \mu_0$ contains only one point, and hence $\mu_0$ is a Dirac mass.

  %       \textcolor{red}{I think the argument in the previous paragraph is a tiny bit simpler than the one we've had before (see the blue text below).}

   %       \textcolor{blue}{Since $\sph$ is compact and $\mu(\sph)=1$, one can find a sequence $\{B_n\}_{n=1}^\infty$ of closed  spherical caps such that $B_{n+1}\subset B_n$, $\mu_0(B_n)>0$ for each $n\in\NN$  and $\lim_{n\to \infty}\text{rad}(B_n)=0$. Let $e\in \bigcap_{n=1}^\infty B_n$. We  claim that $\mu_0 (\sph \setminus B_n)=0$ for each $n\in\NN$. Assume otherwise. Then   $\mu_0(\sph \setminus B_n)>0$ for some $n\in\NN$. By the inner regularity of the Borel measure, there exists a compact set $K_n\subset \sph\setminus B_n$ such that $\mu_0(K_n)>0$. Let $f$ be a nonnegative continuous function on $\sph$ such that $f(x)=1$ for $x\in B_n$ and $f(x)=0$ for $x\in K_n$. Since both $K_n$ and $B_n$ have positive $\mu_0$-measure, this contradicts the fact that every continuous function on $\sph$ is constant $\mu_0$ a.e. on $\sph$. Thus, $\mu_0(\sph\setminus B_n)=0$ for all $n\in\NN$, and
%          \begin{align*}
 %           \mu_0(\sph\setminus\{e\})=\mu_0(\bigcup _{n=1}^\infty B_n^c)=0.
 %         \end{align*}
%This shows that $\mu_0=\da_e$.}

If we assume that \eqref{aa} fails, i.e. $\widehat{F} (n;\lambda)  = 0 $ for some $n \ge 1$ (if  $\widehat{F} (n;\lambda)  < 0 $, then $\sigma$ is not a minimizer  by Proposition \ref{prop-1-2}), then the argument of \eqref{e.perturb} shows that for $d\mu (x) = \big(1 + \varepsilon Y_{n,1} (x)  \big) d\sigma (x)$ we  have $I_F (\mu)= I_F (\sigma)$, i.e. $\sigma$ is not a unique minimizer.
    \end{proof}

We would like to note that  functions with $\wh{F}(n; \ld)\ge 0$  for all $n\ge 1$ are, up to constant terms, {\emph{positive define functions on the sphere}} (see the discussion  in the beginning of \S \ref{s.stol}), which were introduced by Schoenberg \cite{schoen} also in the  context of energy optimization.

In the end of this section  we state some additional results about the extremizers of $I_F$ in terms of the signs of the Gegenbauer coefficients $\wh{F} (n;\lambda)$, which may be proved by identical arguments. These statements could be used to prove  parts \eqref{G3} and \eqref{G4} of Theorem \ref{GeodExt}, which were proved in \cite{BD1} by other means.
\begin{lem}\label{l.aux}
Let $F \in C [-1,1]$ and $\lambda = \frac{d-1}{2}$.
\begin{enumerate}[(i)]
 \item If $(-1)^n \wh{F}(n;\lambda)\ge 0$ for all $n\ge 1$, then
    $$ \max_{\mu\in \mathcal B} I_F (\mu)  = \wh{F}(0;\lambda) +\sum_{n=1}^\infty \wh{F}(2n;\lambda) a_{2n}^d =\f {F(1)+F(-1)}2= I_F \bigg({\frac{\da_{e} +\da_{-e}}{2}}\bigg).$$
Moreover, if $(-1)^n \wh{F}(n;\lambda) >  0$ for all $n\ge 1$, then any maximizer of $I_F$ is of the form $\mu = \frac12 (\da_{e} +\da_{-e})$ for some $e \in \mathbb S^d$.

      \item If $\wh{F}(2n;\lambda)=0$ and $\wh{F}(2n-1; \lambda) \leq  0$ for all $n\in \mathbb N$, then
$$ \max_{\mu \in \mathcal B } I_F (\mu ) = \wh{F}(0; \lambda)$$ and  the maximum  is achieved for any symmetric measure $\mu$ (i.e. $\mu(E)=\mu(-E)$ for all measurable $E\subset \sph$).

\noindent If $\wh{F}(2n;\lambda)=0$ and $\wh{F}(2n-1; \lambda) <  0$ for all $n\in \mathbb N$, then all maximizers  of $I_F$ are symmetric.
\end{enumerate}
\end{lem}

%\begin{rem}If $\int_{-1}^1 |f(t)| (1-t^2)^{\ld-\f12}\, dt <\infty$ and $f$ is continuously differentiable at $x_0\in [-1,1)$,
%then
%$$f(x_0)=\sum_{k=0}^\infty \wh{f}(k) \f {k+\ld}{\ld} C_k^\ld (x_0).$$
%\end{rem}

\section{Geodesic distance Riesz energy integrals}\label{s.gdei}
We now apply the results of the previous section to the specific case of the geodesic distance energy integral \eqref{gdei}.  In order to avoid singularities, we introduce standard modifications of the potentials.   For $t\in [-1,1]$ and $0\leq \va<1$, we define as in \eqref{gdei}
\begin{align*}
   F_{\da,\va} (t)=\begin{cases}(\va+\arccos t)^\da, & \   \  \text{if $\da\neq 0$};\\
   \log\Bl(\f \pi {\va+\arccos t}\Br),&\   \  \text{if $\da=0$}.
   \end{cases}
\end{align*}
We write $F_\da (t)=F_{\da,0}(t)$.
For $\mu\in\mathcal{B}$, define
$$ I_{d, \da}(\mu):=I_{F_\da} (\mu )=\int_{\sph}\int_{\sph} F_\da(x\cdot y)\, d\mu(x)d\mu(y).$$

The main goal in this section is to show the following theorem, which constitutes parts \eqref{G1} and \eqref{G2} of Theorem \ref{GeodExt}.
\begin{thm}\label{thm-2-1}
The normalized Lebesgue measure $d\s$ on $\sph$ is  the unique maximizer for the integral $I_{d, \da}(\mu)$ when  $\da\in (0,1)$, and    is the unique minimizer of $I_{d, \da}(\mu)$ when $-d<\da\leq 0$.
\end{thm}

The following lemma plays a crucial role in the proof of Theorem \ref{thm-2-1}.
\begin{lem}\label{lem-2-1}Let $\va\in [0,1)$ and $\ld>0$.  For $\da>-(2\ld+1)$, define
$$ \mathcal I_{n,\va}^\da:=\int_{-1}^1 F_{\da,\va} (t) C_n^\ld (t) (1-t^2)^{\ld-\f12}\, dt,\  \  n=0,1,\cdots.$$
 Then the following statements hold:
\begin{enumerate}[\rm (i)]
\item\label{ll1} If $\da\in (0,1)$, then $\mathcal I_{n,\va}^\da<0$ for all $n=1,2,\cdots$.
    \item\label{ll2} If $-(2\ld+1)<\da\leq 0$, then $\mathcal I_{n,\va}^\da >0$ for $n=0,1,2,\cdots$.
\end{enumerate}
\end{lem}
\begin{proof}
\eqref{ll1}  By     Rodrigues' formula for ultraspherical polynomials (see, for instance, \cite[4.1.72]{Sz}),
\begin{equation*}\label{2-1}
    C_n^\ld (t)=\f {(-1)^n 2^n}{n!} \f{\Ga(n+\ld)\Ga(n+2\ld)}{\Ga(\ld) \Ga(2n+2\ld)} (1-t^2)^{-(\ld-\f12)} \Bl( \f d {dt}\Br)^n (1-t^2)^{n+\ld -\f12},
\end{equation*}
  and the fact that $C_n^\ld(-t)=(-1)^n C_n^\ld(t)$ for all $t\in [-1,1]$,  it is easily seen   from integration by parts
 that for $k,n=0,1,\cdots$,
\begin{align}\label{2-2}
    \int_{-1}^1 t^k C_n^\ld (t)(1-t^2)^{\ld-\f12}\, dt \begin{cases}
    >0, &  \  \  \text { if $k\ge n$ and $k-n$ is even},\\
    =0,&\   \  \text{otherwise.}
    \end{cases}
\end{align}
Thus, for the proof of assertion \eqref{ll1}, it is sufficient to show that for $\da\in (0,1)$,
\begin{equation}\label{2-3-0}
    F_{\da,\va}(t)=a_0(\va)+\sum_{k=1}^\infty a_k(\va) t^k, \  \  |t|<1\   \ \text{with $ a_k(\va)< 0$ for $ k=1,2,\cdots$}.
\end{equation}
Indeed,
once  \eqref{2-3-0} is proved, then using \eqref{2-2} we obtain that  for $n\ge 1$,
\begin{align*}
    \mathcal I_{n,\va}^\da&=\sum_{k=0}^\infty a_k(\va) \int_{-1}^1t^k C_n^\lambda (t)(1-t^2)^{\ld-\f12}\, dt \notag\\ &=-\sum_{j=0}^\infty |a_{n+2j}(\va)|\int_{-1}^1t^{n+2j} C_{n}^\lambda (t)(1-t^2)^{\ld-\f12}\, dt<0.
\end{align*}

  To show  \eqref{2-3-0}, we use   the  Maclaurin series   of the function $\arccos t$ on the interval $[-1,1]$:
   \begin{equation}\label{2-4}
    \arccos \ t =\f \pi 2 -\sum_{n=0}^\infty \f {\binom{2n}{n}}{4^n (2n+1)} t^{2n+1}=:\f \pi2-A(t),\   \  |t|\leq 1.
   \end{equation}
   The main point in \eqref{2-4} lies in the fact that
\begin{equation}\label{2-5}
  A(t):= \sum_{n=0}^\infty \f {\binom{2n}{n}}{4^n (2n+1)} t^{2n+1}
\end{equation}
   is an odd power series  with positive coefficients.
Clearly,
   \begin{align}\label{2-6}
     |A(t)|=\Bl|\f \pi 2-\arccos t\Br|< \f \pi2,\  \ t\in (-1,1).
   \end{align}
 Thus,   using \eqref{2-4} and \eqref{2-6}, we obtain that for $t\in (-1,1)$ and $\va\in [0,1)$,
   \begin{align*}
    F_{\da,\va} (t) =\Bl(\f \pi 2+\va\Br)^\da \Bl(1-\f {2A(t)}{\pi+2\va} \Br)^\da=\Bl(\f \pi 2+\va\Br)^\da+\Bl(\f \pi {2+\va}\Br)^\da\sum_{j=1}^\infty b^\da_j \Bl(\f {2 A(t)}{\pi+2\va}\Br)^j,
   \end{align*}
   where
   $$ b_j^\da =\f {(-1)^j \da(\da-1)\cdots(\da-j+1)}{j!}=-\f { \da(1-\da)\cdots(j-1-\da)}{j!}.$$
Clearly, each $b_j^\da$ ($j\ge 1$)  is negative for $\da\in (0,1)$.
Then  \eqref{2-3-0}  follows from \eqref{2-5}.\\

\eqref{ll2} As in the proof of assertion  \eqref{ll1} ,
  it suffices  to show that $F_{\da,\va}(t)$ has a Maclaurin series representation with positive  coefficients on the interval $[-1,1]$. For $\da=0$, we use the Maclaurin series    of the function $\log (1-t)$ on the interval $(-1,1)$:
$$\log (1-x)=-\sum_{n=1}^\infty \f {x^n}n,\  \ |x|<1.$$
We then obtain from \eqref{2-6}  that for $t\in (-1,1)$,
\begin{align}
     F_{0,\va}(t)&=\log \f{2\pi}{\pi+2\va} -\log\Bl(1-\f {2 A(t)}{\pi+2\va}\Br)\notag\\
     &=\log \f{2\pi}{\pi+2\va} +\sum_{n=1}^\infty \f {1}{n}\Bl(\f {2A(t)}{\pi+2\va}\Br)^n.\label{2-7}
    \end{align}
For $\da=-s<0$, we have
\begin{align}\label{2-8}
    F_{\da,\va}(t)&=\Bl(\f \pi2+\va-A(t)\Br)^{-s}\notag\\
    &=\Bl(\f 2{\pi+2\va}\Br)^{-s} \sum_{j=0}^\infty \f {s(s+1)\cdots (s+j-1)}{j!} \Bl( \f {2A(t)}{\pi+2\va}\Br)^j.
\end{align}
 Combining \eqref{2-7}, \eqref{2-8} with \eqref{2-5}, we conclude that if $\da\in (-2\ld-1, 0]$ and $\va\in [0,1)$, then  all the coefficients of the Maclaurin series of the function  $F_{\da,\va}(t)$  are positive. Assertion \eqref{ll2}  then follows by \eqref{2-2}.
\end{proof}

We are now in a position to show Theorem \ref{thm-2-1}.

\begin{proof}[Proof of Theorem \ref{thm-2-1}] For $\da\in (0,1)$, the function $F_\da(t)$ is continuous on $[-1,1]$, and hence the stated assertion follows directly from Proposition \ref{prop-1-2} and  part \eqref{ll1} of Lemma \ref{lem-2-1}.

For $-d<\da\leq 0$, the function $F_\da$ is not continuous at $t=1$ and, therefore, we need a slight modification of the proof.  For the moment, we assume that  $\da\neq 0$ and write $\da=-s$ with $0< s<d$.
Recall that for each $\va\in (0,1)$,
$$ F_{\da, \va}(t)=\Bl( \arccos t+\va\Br)^{-s},\   \   \  t\in [-1,1],$$
 and by part \eqref{ll2} of  Lemma \ref{lem-2-1},
$$\wh{F_{\da, \va}}(n;\ld)=c_n\int_{-1}^1 F_{\da,\va}(t) C_n^\ld(t) (1-t^2)^{\ld-\f12}\, dt >0,\   \  n=0,1,\cdots.$$
 Using Proposition \ref{prop-1-2}, we conclude  that  $I_\mu (F_{\da,\va})$ has a unique   minimizer $d\s$. Hence,    for any $\mu\in \mathcal{B}$ and any $\va>0$,
 \begin{align*}
  \int_{\sph}\int_{\sph} \big( \rho (x, y)\big)^{-s}\, d\mu(x)\, d\mu(y)&\ge   \int_{\sph}\int_{\sph} \big( \va + \rho(x, y)\big)^{-s}\, d\mu(x)\, d\mu(y)\\
  &\ge \int_{\sph}\int_{\sph} \big(\va + \rho(x, y)\big)^{-s}\, d\sa(x)\, d\sa(y).
\end{align*}
Letting $\va\to 0$, and using the monotone convergence theorem, we get
$$\int_{\sph}\int_{\sph} \big(\rho(x, y)\big)^{-s}\, d\mu(x)\, d\mu(y)\ge \int_{\sph}\int_{\sph} \big(\rho(x, y)\big)^{-s}\, d\sa(x)\, d\sa(y),\  \  \forall \mu\in\mathcal{B}.$$
This shows that $\s$ is a minimizer of $I_{F_\da} (\mu) $ for $0<s=-\da<d$.

 Next, we show  the minimizer is unique. Let $\mu_0\in\mathcal{B}$. If  $d\mu_0\neq d\s$, then  there must exist  a spherical harmonic $P$ of degree $n_0\ge 1$ such that
$\int_{\sph} P(x)\, d\mu_0(x)\neq 0$.  By  \eqref{1-5}, this implies that
\begin{align*}
    b_{n_0,\mu_0}&=\f {n_0+\ld}{\ld}\int_{\sph}\int_{\sph}C_{n_0}^\ld(x\cdot y)\, d\mu_0(x)\, d\mu_0(y)\\
    &\ge \f 1{\|P\|_2^2} \Bl(\int_{\sph} P(x)\, d\mu_0(x)\Br)^2 \ge c>0.
\end{align*}
 However, according to \eqref{1-3-1}, we have that for any $\va\in (0,1)$,
$$ I_{F_\da} ({\mu_0} )  \ge I_{F_{\da,\va} } ({\mu_0})  \ge \wh{F_{\da, \va}}(0;\ld) + \wh{F_{\da, \va}}(n_0;\ld) b_{n_0,\mu_0}.$$
Letting $\va\to 0$, we obtain from part \eqref{ll2} of  Lemma \ref{lem-2-1} (ii) that
$$ I_{F_\da} (\mu_0)\ge I_{F_\da} (\sa ) + \wh{F_{\da}}(n_0;\ld) b_{n_0, \mu_0}>I_{F_\da} (\sa ).$$
Since $\mu_0$ is an arbitrary Borel probability measure on $\sph$, this shows the uniqueness of the minimizer. Finally, we point out that the above proof with a slight modification works equally well for the case of $\da=0$.
\end{proof}

 The methods employed here are quite standard in the context of energy optimization on the sphere (see, e.g.,  \cite{KS} for the classical Riesz energy). While Theorem \ref{thm-2-1} covers parts \eqref{G1} and \eqref{G2} of Theorem \ref{GeodExt}, the remaining two cases  (proved in \cite{BD1}) could be proved in the same way, using the above computations and results of Lemma \ref{l.aux}.

Notice that we have, in fact, used Taylor expansions of the underlying function $F$ in order to obtain information about the signs of the ultraspherical coefficients. Recently, after the first author's  presentation of the results of this paper and \cite{BD1}, Y.S. Tan \cite{Tan} found a beautiful alternative proof of parts \eqref{G2}-\eqref{G4} of Theorem  \ref{GeodExt}, which involves only Taylor series (does {\emph{not}} resort to the use of spherical harmonics) and uses an interesting ``tensorization trick". While this method is somewhat less general than the one presented in \S \ref{prelim} (since it requires $F$ to be analytic), it is applicable to a variety of natural situations (including Bjorck's theorem \cite{Bjorck}).

\section{ Discrepancy and Stolarsky principle}\label{s.stol}

Let us denote by $\Phi_d$ the set of all continuous functions $F$ on $[-1,1]$ for which $\wh{F}(n; \ld)\ge 0$ for all $n \in \mathbb N \cup \{ 0 \}$, where, as before, $\ld=\f{d-1}2$. Functions in the class $\Phi_d$ are known as {\emph{positive definite}} functions on the sphere and, up to constants, are precisely the functions discussed in \S \ref{prelim} (Proposition \ref{prop-1-2}). Their connection to energy optimization is well known \cite{schoen}. There is a variety of characterizations of the class $\Phi_d$, but we shall particularly use  the following.

\begin{lem} A function $F\in \Phi_d$ if and only if there exists a function $f\in L^2_{w_\ld}[-1,1]$ such that
\begin{equation}\label{3-1}
    F(x\cdot y)=\int_{\sph} f(x\cdot z) f(z\cdot y)\, d\s(z),\   \ x, y\in \sph.
\end{equation}
 \end{lem}
\begin{proof} The {\it sufficiency}
  part  is obvious. Indeed, if \eqref{3-1} holds  for some $f\in L_{w_\ld}^2[-1,1]$, then $F$ is continuous and $\wh{F}(n; \ld)=|\wh{f}(n,\ld)|^2\ge 0$ for all $n=\mathbb N \cup \{ 0\}$.

It remains to show the {\it necesity.}
  Let $F$ be a continuous function on $[-1,1]$  such that $\wh{F}(n;\ld)\ge 0$ for all $n\ge 0$.
Define $$ f(t)=\sum_{n=0}^\infty \sqrt{\wh{F}(n;\ld)} \f {n+\ld}{\ld} C_n^\ld (t),\ \  t\in [-1,1].$$
$f$ is a well defined function in $L^2_{w_\ld}[-1,1]$ since,
by Plancherel's formula and \eqref{1-1-a},
\begin{align*}
    c_d\int_{-1}^1 |f(t)|^2 (1-t^2)^{\ld-\f12}\, dt =\sum_{n=0}^\infty \wh{F}(n;\ld)\f {n+\ld}{\ld} C_n^\ld (1)<\infty.
\end{align*}
Furthermore, \eqref{3-1} holds since $\wh{F}(n; \ld)=|\wh{f}(n,\ld)|^2$ for all $n\ge 0$.
   \end{proof}

   For the rest of this section, we will assume that $F\in\Phi_d$, and  $f\in L_{w_\ld}^2[-1,1]$ is chosen so that   \eqref{3-1} is satisfied.

Given  a finite set of points  $Z=\{z_1, \cdots, z_N\}\subset \sph$, we define its $L^2$  discrepancy with respect to a function  $f: [-1,1] \rightarrow \mathbb R$  by
\begin{align*}
    D_{L^2, f}(Z) =\Bl(\int_{\sph}\Bl|\int_{\sph} f(x\cdot y)\, d\s(y)-\f 1N \sum_{j=1}^N f(x\cdot z_j)\Br|^2\, d\s(x)\Br)^{\f12}.
\end{align*}
We define the optimal $L^2$ discrepancy  by setting
$$\mathcal{D}_{L^2, f,N}=\inf_Z D_{L^2, f}(Z),$$
where the infimum is taken over all $Z\subset \sph$ with $\# Z=N$. The discrepancy $D_{L^2, f}(Z)$  measures the uniformity of the finite distribution of points $Z$ with respect to the function $f$. If one takes, for example, $f (\tau) = {\bf 1}_{[1-t,1]} (\tau)$, one obtains the well-studied spherical cap discrepancy, see \eqref{e.sphcap}.

The link between discrepancy and energy on the sphere has been first established by Stolarsky \cite{stol} who established an identity relating the spherical cap $L^2$ discrepancy and the sum of pairwise Euclidean distances between the points of $Z$. Identities of this type came to be called {\it{Stolarsky invariance principle}}. There has been an increase of activity on this subject in the recent years \cite{Br,BrDick,owen,skrig,skrig2,skrig3,BL}. In our companion paper \cite{BD1} we explore a number of variations of this principle and its applications to energy optimization, in particular, part \eqref{G3} of Theorem \ref{GeodExt}.

 Here we present a general form of the Stolarsky principal, which relates the discrepancy $D_{L^2, f}(Z)$, discrete energy $E_F (Z)$, and the energy integral $I_F (\sigma)$. We also apply this principle to  estimating the optimal discrepancy $\mathcal{D}_{L^2, f,N}$.

\begin{thm}\label{thm-3-2}  Let $\ld=\f{d-1}2$. Assume that  $F \in C[-1,1]$ and $f\in L_{w_\ld}^2[-1,1]$ satisfy relation \eqref{3-1}. \begin{enumerate}[\rm (i)]
\item\label{St1} (Stolarsky principle) Given a set of $N$-points $Z=\{z_1,\cdots, z_N\}\subset \sph$,
\begin{align}\label{e.stol}
    N^{-2} \sum_{i=1}^N \sum_{j=1}^N F(z_i\cdot z_j) =D^2_{L^2, f}(Z) +\int_{\sph}\int_{\sph} F(x\cdot y)\, d\s(x) d\s(y).
\end{align}
\item\label{St2}  There exist  constants $c_d,\, C_d >0$, such that  for any  $N\in\NN$,
\begin{align}
   C_d \min_{1\leq k\leq c_d N^{1/d}} \wh{F}(k,\ld)\leq   \mathcal{D}_{L^2, f, N}^2\leq  N^{-1}\max_ {0\leq t\leq c_d' N^{-\f1d}} \bl( F(1)-   F(\cos t)\br).\label{e.discrest}\end{align}
%\item If, in addition, $F$ is nonnegative, then \begin{equation}\label{}
%    \mathcal{D}_{L^2, f, N}^2\ge \f {F(1)}N -\wh{F}(0; \ld).
%\end{equation}
\end{enumerate}
\end{thm}

We make a few remarks before proceeding to the proof of this theorem. First of all, notice that \eqref{e.stol} implies the minimizing the discrete energy $E_F (Z)$ is equivalent to minimizing the   discrepancy $D_{L^2, f}(Z)$. Moreover, the square of this  discrepancy yields the difference between the discrete energy and the optimal energy integral -- a quantity which will be investigated deeper in the following section, \S\ref{s.disc}. We also observe that, while this approach is rather general, it is somewhat indirect, since it is not easy to explicitly find the function $f$ for a given $F$ (and vice versa). However, since $\wh{F}(n; \ld)=|\wh{f}(n,\ld)|^2$, the lower bound in \eqref{e.discrest} may be used to estimate both discrepancy and energy.

\begin{proof}  \eqref{St1}  This identity  can be  verified by a direct computation involving relation \eqref{3-1}. In fact,  an even more general form is proved in Theorem 5.10 of our parallel paper \cite{BD1}, which, in particular, gives an alternative proof of the fact that $\sigma$ minimizes $I_F$ for $F \in \Phi_d$.

\eqref{St2} We start with the proof  the upper estimate:
\begin{equation}\label{3-2}
     \mathcal{D}_{L^2, f, N}^2\leq  N^{-1}\max_ {0\leq t\leq c_d N^{-\f1d}} \bl( F(1)-   F(\cos t)\br).
\end{equation}
The   proof follows along the same lines as that of Theorem 1 of \cite{KS}. Let
$\{R_1, \cdots, R_N\}$ be a  partition of  $\sph$ such that (see, for instance, \cite[Sec. 6.4, p. 140]{DX})
$$
     \s(R_j)=\f1N \quad \text{and} \quad \text{diam} (R_j) \leq c_d N^{-\f1{d}},\  \  j=1,2,\cdots, N.
$$
 Denote by $\s_j^\ast$  the restriction of the measure $N {\s}$ to $R_j$, and
let
$\Omega^N=R_1\times \cdots \times R_N$ denote the product measure space with probability measure $d\s_1^\ast\times \cdots \times d\s_N^\ast$.
 Then
 \begin{align}
    \mathcal{D}_{L^2, f, N}^2
    &\leq
    N^{-2}\int_{\sph}\int_{\Omega^N}\Bl| \sum_{j=1}^N \Bl[f(x\cdot z_j)-\int_{R_j} f(x\cdot z)\, d\s_j^\ast(z)\Br]\Br|^2\times \notag\\
   & \   \hspace{30mm}\times d\s_1^\ast(z_1)\cdots d\s_N^\ast(z_N)\, d\s(x)\notag\\
       &= N^{-2}\int_{\sph}\sum_{j=1}^N \Bl[\int_{R_j} |f(x\cdot z)|^2 \, d\s_j^\ast(z)-\Bl(\int_{R_j} f(x\cdot z)\, d\s_j^\ast(z)\Br)^2\Br]\, d\s(x)\notag\\
    &=N^{-1}F(1)-\sum_{j=1}^N \int_{\sph} \Bl(\int_{R_j} f(x\cdot z)\, d\s(z)\Br)^2 \, d\s(x).\label{3-3}\end{align}
    Note, however, that  for each $1\leq j\leq N$,
    \begin{align}
\label{3-4}        \int_{\sph} \Bl(\int_{R_j} f(x\cdot z) & \, d\s(z)\Br)^2\, d\sigma (x) =\int_{\sph} \int_{R_j}\int_{R_j} f(x\cdot z) f(x\cdot y)\, d\s(z)\, d\s(y)\, d\s(x)\notag\\
         &=\int_{R_j}\int_{R_j} F(y\cdot z)\, d\s(y)\, d\s(z)\ge N^{-2} \min_{0\leq t\leq c_d N^{-\f1d}} F(\cos t).
    \end{align}
   Combining  \eqref{3-3} with \eqref{3-4}, we deduce  estimate \eqref{3-2}.

   Next, we prove the lower estimate:
   \begin{equation}\label{3-6-0}
     \mathcal{D}_{L^2, f, N}^2\ge  C_d \min_{1\leq k\leq c_d N^{1/d}} \wh{F}(k,\ld).
   \end{equation}

   Let $a>1$ be a large parameter depending only on $d$ such that $n:=a N^{1/d}$ is an integer.
Let
$K_{n}(t)$ denote the Ces\`aro kernel of order $d+1$ for the spherical harmonic expansions on $\sph$; that is,
$$ K_n(t)=\sum_{k=0}^n \f {A_{n-k}^{d+1}}{A_n^{d+1}} \f{k+\ld}{\ld} C_k^\ld(t),\   \ t\in [-1,1].$$
It is known that (see, for instance, \cite{BC} and \cite[Theorem 7.6.1, p. 389]{AAR})
$$ 0\leq K_{n}(\cos \t)\leq C n^d (1+n\t)^{-d-1},\   \  \forall \t\in [0,\pi].$$

We claim that for each $Z_N=\{z_1,\cdots, z_N\}\subset \sph$,
\begin{equation}\label{3-7-0}
    \int_{\sph}\Bl|1-N^{-1}\sum_{j=1}^N K_n(x\cdot z_j)\Br|^2\, d\s(x)\ge c_d>0.
\end{equation}
To see this, we first note that   Bernstein's inequality for trigonometric polynomials implies that
\begin{equation}\label{3-8-0}
  \min_{0\leq \t \leq \f 1{2n}}K_{n}(\cos \t)\ge \f12 K_{n}(1) =\f 12 \|K_n\|_\infty \ge c_d n^d.
\end{equation}
Next, let $\{R_1,\cdots, R_{ N_1}\}$ be an area-regular partition of $\sph$  such that $N_1=C_{d,a}N$, $\sa(R_j)=\f 1{N_1}$ and  $\text{diam} (R_j)\leq \f 1{2n}$ for $1\leq j\leq N_1$.
Set  $$\Ld:=\Bl\{j:\  \ 1\leq j\leq N_1,\   \ R_j\cap Z_N\neq \emptyset\Br\}.$$
 Using \eqref{3-8-0} and  positivity of the kernel $K_n$, we have that for each  $x\in R_j$ with  $j\in\Ld$,
\begin{align*}
   \f 1N \sum_{i=1}^N K_{n}(x\cdot z_i)& \ge \f 1N \sum_{z\in Z_N\cap R_j} K_{n}(x\cdot z)
   \ge c_d \f {n^d}N \cdot \# \{R_j\cap Z_N \} \\
   &=c_d a^d \cdot \# \{ R_j\cap Z_N\} >2\#\{ R_j\cap Z_N\},
\end{align*}
provided that the parameter $a$ is large enough.
It then follows that
\begin{align*}
    \int_{\sph} &\Bl|1-N^{-1} \sum_{i=1}^N K_{n} (x\cdot z_i)\Br|^2\, d\s(x)
    \ge \sum_{j\in\Ld}\int_{R_j} \Bl|1-\f 1N \sum_{i=1}^N K_{n} (x\cdot z_i)\Br|^2\, d\s(x)\\
    &\ge \f 1{N_1}\sum_{j\in\Ld} |\#(R_j\cap Z_N)|^2 \ge \f 1{N_1}\sum_{j\in\Ld} \#(R_j\cap Z_N)=\f N{N_1}\ge \f1{C_{d,a}}>0.
\end{align*}
This proves the claim \eqref{3-7-0}. We are now ready  to show the lower estimate \eqref{3-6-0}.  Recall that \begin{align*}
    f(x\cdot e) =\sum_{k=0}^\infty \wh{f}(k;\ld) \f {k+\ld}{\ld} C_k^\ld(x\cdot e),\    \  \forall e\in\sph,
\end{align*}
where the series converges in  the norm of $L^2(\sph)$,
and
$\wh{f}(0;\ld)=\int_{\sph} f(x\cdot z)\, d\s(z)$ for any $x\in\sph$. Thus, by orthogonality of spherical harmonics, we obtain
\begin{align*}
    \int_{\sph}&\Bl|\int_{\sph} f(x\cdot z)\, d\s(z)-N^{-1} \sum_{j=1}^N f(x\cdot z_j)\Br|^2\, d\s(x)\\
    & =\int_{\sph} \Bl|\sum_{k=1}^\infty \wh{f}(k,\ld)\f {k+\ld}\ld  N^{-1} \sum_{j=1}^N C_k^\ld (x\cdot z_j)\Br|^2\, d\sa(x)\\
    &=\sum_{k=1}^\infty \wh{F}(k,\ld)\Bl\|\f {k+\ld}\ld  N^{-1} \sum_{j=1}^N C_k^\ld (\la z_j,\cdot\ra)\Br\|_2^2.
    \end{align*}
    Since $\f {A_{n-j}^{d+1}}{A_{n}^{d+1}}\leq 1$ for all $0\leq j\leq n$, it follows that \begin{align}\label{3-9-0}
    \int_{\sph}&\Bl|\int_{\sph} f(x\cdot z)\, d\s(z)-N^{-1} \sum_{j=1}^N f(x\cdot z_j)\Br|^2\, d\s(x) \notag\\
    &\ge   \sum_{k=1}^n \wh{F}(k,\ld)\Bl|\f{A_{n-k}^{d+1}}{A_n^{d+1}}\Br|^2\Bl\|\f {k+\ld}\ld  N^{-1} \sum_{j=1}^N C_k^\ld (\la z_j,\cdot\ra)\Br\|_2^2,
    \end{align}
which, using \eqref{3-7-0}, is bounded below by
    \begin{align*}
          \Bl(\min_{1\leq k\leq n} & \wh{F}(k,\ld)\Br)\int_{\sph} \Bl|  1-N^{-1} \sum_{j=1}^N K_{n} (x\cdot z_j)\Br|^2\, d\s(x)\\
    &\ge C_d \min_{1\leq k\leq n} \wh{F}(k,\ld).
    \end{align*}
      This yields the desired lower estimate \eqref{3-6-0}.
\end{proof}

Using Theorem \ref{thm-3-2}, one  can give a new simpler proof of  a well-known result of  Beck \cite{Beck} regarding the lower estimate of the spherical cap discrepancy:
\begin{equation}\label{e.sphcap}
\DL(Z_N)^2:=\int_{-1}^1 \int_{\sph} \Bl| \f {\# (Z_N\cap B(x,t))}N -\sa(B(x,t))\Br|^2 \, d\s(x)\, dt,
\end{equation}
where  $Z_N:=\{z_1,\cdots, z_N\}$ is a set of $N$-distinct points on $\sph$ and
$$B(x,t):=\{y\in\sph:\  \ x\cdot y\ge t\},\   \ x\in\sph,\  \ t\in [-1,1].$$

\begin{cor}\label{beck} \cite[J.~Beck, 1984]{Beck} Given an arbitrary set $Z_N$ of $N$-distinct points on the sphere $\sph$,
$$\DL(Z_N)\ge C_d N^{-\f12-\f1{2d}}.$$
\end{cor}
\begin{proof}

Let $f_t(s)=\chi_{[t,1]}(s)$ for $t, s\in [-1,1]$. Using the formula,
$$\f {d}{dx} \Bl( C_{n-1}^{\ld+1}(x) (1-x^2)^{\ld+\f12}\Br)=-\f {n(n+2\ld)}{2\ld} C_n^\ld (x) (1-x^2)^{\ld-\f12},$$
we have
\begin{align*}
    \wh{f_t}(n;\ld)&=c_\ld\f{\Ga(n+1)}{\Ga(n+2\ld)}\int_t^1 C_n^\ld (x)(1-x^2)^{\ld-\f12}\, dx
    =c_d (1-t^2)^{\ld+\f12} R_{n-1}^{\ld+1} (t).
\end{align*}
This implies that
\begin{align}
    \int_{-1}^1 |\wh{f_t}(n;\ld)|^2 \, dt \sim \int_{-1}^1 \bl|R_{n-1}^{\ld+1}(t)\br|^2 (1-t^2)^{2\ld+1}\, dt\sim n^{-2\ld-2}=n^{-d-1},
\end{align}
where the second step uses the known estimates on integrals of Jacobi polynomials (see, for instance, \cite[ Ex.  91, p. 391]{Sz}).

On the other hand, using \eqref{3-9-0} and \eqref{3-7-0}, with $n\sim N^{1/d}$,  we have
\begin{align*}
   \DL(A)^2&= \int_{-1}^1 D_{L^2, f_t}(Z_N)^2\, dt\\
   &\ge
      \sum_{k=1}^n \int_{-1}^1 |\wh{f_t}(k,\ld)|^2\, dt\Bl|\f{A_{n-k}^{d+1}}{A_n^{d+1}}\Br|^2\Bl\|\f {k+\ld}\ld  N^{-1} \sum_{j=1}^N C_k^\ld (\la z_j,\cdot\ra)\Br\|_2^2\\
      &\ge  \Bl(\min_{1\leq k\leq n}  \int_{-1}^1  |\wh{f_t}(k,\ld)|^2\, dt\Br)\int_{\sph} \Bl|  1 -N^{-1} \sum_{j=1}^N K_{n} (x\cdot z_j)\Br|^2\, d\s(x)\\
    &\ge C_d \min_{1\leq k\leq n} \int_{-1}^1 |\wh{f_t}(k,\ld)|^2\, dt\ge c_d n^{-d-1}\sim N^{-1-\f1d}.
    \end{align*}
\end{proof}

%
%
%\begin{rem}One can also show the following estimate:
%\begin{equation}\label{3-5}
%   \mathcal{D}_{L^2, f, N}^2\leq C_d\sum_{j\ge c_d N^{\f 1d}} \wh{F}(j;\ld) j^{d-1}.
%\end{equation}
%  Indeed,    using  a recent result from \cite{BRV} on spherical $n$-designs, we can find  a set $Z_N=\{z_1, \cdots, z_N\}$ of $N$-distinct points on $\sph$ such that
%     \begin{equation}\label{3-6}
%     \int_{\sph} P(x)\, d\s(x)=\f 1N \sum_{j=1}^N P(z_j),\   \ \forall P\in\Pi_n^{d+1},\  \  \text{with $ n=c_d N^{\f1d}$}.
%     \end{equation}
%Setting
%$$ S_n f(t)=\sum_{j=0}^n \wh{f}(j;\ld) \f {j+\ld}{\ld} C_j^\ld (t),$$
% We deduce from \eqref{3-6},  that for any $x\in \sph$,
%\begin{align*}
%&\Bl|\int_{\sph} f(x\cdot y)\, d\s(y)-\f 1N \sum_{j=1}^N f(x\cdot z_j)\Br|
%=\f1N\Bl|\sum_{j=1}^N \bl(f(x\cdot z_j)-S_nf(x\cdot z_j)\br)\Br|.
%\end{align*}
%It follows by the Minkowskii inequality  that
%\begin{align*}
%    D_{L^2, f}(Z_N)\leq C \|f-S_n f\|_{2,\ld} \leq C \Bl(\sum_{j\ge c_d N^{\f1d}} |\wh{f}(j;\ld)|^2 j^{2\ld}\Br)^{1/2},
%\end{align*}
%which, combined with the fact that $\wh{F}(n;\ld)=|\wh{f}(n;\ld)|^2$, implies \eqref{3-5}.
%\end{rem}

\section{Discrete Riesz  energy}\label{s.disc}

Now that we understand that the energy $I_{d,\delta} (\sigma)$ is optimal for $-d<\delta < 1$, it is natural to investigate how well it can be approximated by discrete distributions.  We define the appropriate discrete energy, in accordance with \eqref{DE}.
\begin{defn}
For $\da>-d$ and $\da\neq 0$,  define
  the discrete  $\da$-energy of  a finite subset $Z_N=\{z_1,\cdots, Z_N\}$
of $N$ distinct points on $\sph$  by
 $$E_{d,\da}( Z_N):=\sum_{1\leq i<j\leq N} \rho(z_i, z_j)^{\da},
$$
where $\rho(x,y)=\arccos (x\cdot y)$ is the geodesic distance between $x$ and $y$ on $\sph$.
The discrete $N$-point  $\da$-energy of  $\sph$ is defined by
\begin{equation}\label{1-1-energy-ch14}
  \mathcal{E}_{d,\da}(  N) :=\inf_{Z_N} E_{d,\da} ( Z_N),
\end{equation}
where the infimum is taken over all $N$-point subsets of $\sph$ (and infimum is replaced by supremum for $\delta >0$).
\end{defn}

Notice that, when $\delta >0 $, we have $$ I_{d,\delta} \big( N^{-1} \sum_{i=1}^N \delta_{z_i} \big ) = \frac{2}{N^2}  E_{d,\delta} (Z_N) ,$$ while this energy integral is infinite for $\delta \le 0$ because of diagonal terms.

Our main result in this section can be stated as follows:

\begin{thm}\label{thm-1-4-energy-ch14}
Let $d\ge 2$. If $-d<\da<1$ and  $\da\neq  0$, then
\begin{equation}\label{e.difference}
      I_{d,\da} (\sa) - \frac{2}{N^2} \mathcal{E}_{d,\da}  (  N)\sim  N^{-1-\f \da d}.
\end{equation}
In the logarithmic case, $\delta =0$, we have the estimate
\begin{equation}\label{e.difference0}
      I_{d,0} (\sa) - \frac{2}{N^2} \mathcal{E}_{d,0} (  N)\sim N^{-1} \log N.
\end{equation}
\end{thm}

\begin{rem}For the discrete $\da$-Riesz energy  defined with respect to the Euclidean distance on $\sph$, similar results were previously  established for $-d<\da<2$  in a series of papers (see  \cite[Proposition 2]{BHS},  \cite{Brauchart,RSZ,KS,wagner1,wagner2}).
\end{rem}

Recall that  for $\da>-d$ and $\da\neq 0$,
 \begin{align*}
    I_{d,\da}(\sa)=\int_{\sph}\int_{\sph} \rho(x,y)^\da \, d\s(x)d\s(y)=\f{\Ga(\f {d+1}2)}{\Ga(\f d2)\Ga(\f12)}\int_0^\pi \t^\da \sin^{d-1}\t\, d\t.
\end{align*}
 The following lemma is needed  in the proof of the lower estimates in  Theorem \ref{thm-1-4-energy-ch14}:

\begin{lem}\label{lem-4-3}  If  $\ld>0$ and  $0<\da<1$,  then
\begin{align*}
    -\int_0^\pi \t^\da R_n^\ld(\cos \t) (\sin\t)^{2\ld}\, d\t\sim n^{-2\ld-1-\da},\   \   \  n=1,2,\cdots,
\end{align*}
where $R_n^\ld(t)=\f{C_n^\ld(t)}{C_n^\ld(1)}$.
\end{lem}

We postpone the proof of Lemma \ref{lem-4-3} to  the next section. For the moment, we take it for granted, and proceed with the proof of Theorem \ref{thm-1-4-energy-ch14}.
\subsection{Proof of Theorem \ref{thm-1-4-energy-ch14} for $0<\da<1$}
  %Let $0<\da<1$.  
  By  part \eqref{ll1} of Lemma \ref{lem-2-1}, the function
$F(t)=\Bl(\f \pi2\Br)^\da -(\arccos t)^\da$ belongs to the class $\Phi_d$ for $\da\in (0,1)$.  Furthermore, by Lemma  \ref{lem-4-3}, we have that
\begin{equation}\label{3-17}
    \wh{F}(k; \ld)=-c_\ld \int_0^\pi \t^\da R_k^\ld(\cos\t) (\sin\t)^{2\ld}\, d\t \sim k^{-d-\da},\   \  k=1,2\cdots.
\end{equation}

 Hence, applying the Stolarsky principle (part \eqref{St1} of Theorem \ref{thm-3-2}),  we find  that for
$Z_N=\{z_1,\cdots, z_N\}\subset \sph$,
\begin{align*}
  \frac2{N^{2}} E_{d,\da} (  Z_N)&=  \f 1{N^2}\sum_{i=1}^N \sum_{j=1}^N\rho(z_i, z_j)^\da =I_{d,\da}(\sa) -D_{L^2, f}(Z_N)^2,
 \end{align*}
 and hence
 \begin{align*}
 I_{d,\da}(\sa)- \frac{2}{N^2} E_{d, \da} (  Z_N)& =   D^2_{L^2, f}(Z_N).
 \end{align*}
By  part \eqref{St2} of Theorem \ref{thm-3-2}, we have
  \begin{align*}
 \inf_{Z_N} D_{L^2, f}(Z_N)^2\leq  c_d N^{-1} \max_{0\leq \t\leq c N^{-\f 1d}}
     \t^\da\leq c_d N^{-1-\f\da d},
\end{align*}
whereas by \eqref{3-17} and   part \eqref{St2} of  Theorem \ref{thm-3-2}, for any $Z_N$,
\begin{align*}
 D_{L^2, f}(Z_N)^2\ge c  \min_{1\leq k \leq c_d N^{\f1d}} \wh{F}(k; \ld)\ge c_d N^{- 1-\f \da d}.
\end{align*}
This completes the proof of the theorem for $0<\delta<1$.

\subsection{Proof of  Theorem \ref{thm-1-4-energy-ch14} for $-d<\da<0$}
 For convenience, we  set  $\da=-s$ with $0<s<d$.    We start with the proof of the upper estimate:
\begin{equation}\label{4-4-0}
   \f 2{N^2}  \mathcal{E}_{d,-s} ( N) \leq I_{d,\da}(\sa)-c_d N^{-1-\f \da d}.
\end{equation}
Let  $\va\in (0,1)$, and set
$$F_{-s, \va}(t)=(\arccos t+\va)^{-s},\   \  t\in [-1,1].$$
 Then according to  Lemma \ref{lem-2-1},  $F_{-s,\va}\in\Phi_d$.
Thus, by the Stolarsky principle (part \eqref{St1} of Theorem \ref{thm-3-2}),
\begin{align}
    2 N^{-2} &\sum_{1\leq i<j\leq N}(\rho(z_i, z_j)+\va)^{-s} +N^{-1} \va^{-s}\label{3-11}\\
     &=D^2_{L^2, f_{-s,\va}}(Z_N) +\int_{\sph}\int_{\sph} (\rho(x,y)+\va)^{-s}\, d\s(x)\, d\s(y).\notag
\end{align}
%where $f_{s,\va}\in L^2_{w_\ld}[-1,1]$ satisfies
%$$F_{s,\va}(x\cdot y)=\int_{\sph}f_{s,\va} (x\cdot z) f_{s,\va}(y\cdot z)\, d\s(z),\  \ x, y\in\sph.$$
 We then use \eqref{3-11} and   part \eqref{St2} of Theorem \ref{thm-3-2}  to obtain
\begin{align*}
 \inf_{Z_N}     \sum_{1\leq i<j\leq N} (\rho(z_i, z_j)+\va)^{-s} &=\f 12 N^2 \, \mathcal{D}^2_{L^2, f_{-s,\va},N}  -\f 12 N \va^{-s} \\ & \,\,\,\,\,\,\,  +\f 12 N^2\int_{\sph}\int_{\sph} (\rho(x,y)+\va)^{-s}\, d\s(x)d\s(y)\\
    &\leq \f N 2 \Bl( F_{-s,\va}(1)-\min_{|\theta| \leq c_d N^{-\f1d}} F_{-s,\va}(\cos \theta)\Br) \\ & \,\,\,\,\,\,\,\, - \f N2\va^{-s} +\f {N^2} 2 I_{d,\da}(\sa)\\
    &=-\f 12 N \min_{0\leq t\leq c_d N^{-\f1d}} (t+\va)^{-s} +\f {N^2}2 I_{d,\da}(\sa)\\
    &=-\f 12 N  (c_d N^{-\f1d}+\va)^{-s} +\f {N^2}2 I_{d,\da}(\sa).
\end{align*}
Letting $\va\to 0$ yields the upper estimate \eqref{4-4-0}.\\

Next, we show the lower estimate. Let
 $Z_N=\{z_1,\cdots, z_N\}$ be an arbitrary    set of $N$-distinct points on $\sph$. We need to prove that
\begin{equation}\label{4-6-0}
    E_{d,-s} (Z_N) =\sum_{1\leq i<j\leq N} \rho(z_i, z_j)^{-s}  \ge \f {N^2}2 I_{d,-s}(\sa)-c_d N^{1+\f s d}.
\end{equation}

To this end, let $k$ be the smallest  positive integer such that $ s+k+1>d$.
For a fixed $\t\in (0,\pi]$, define $g_\t :[0, \pi]\to \RR$ by
$$ g_\t (t)=(\t +t)^{-s} + st (\t+t)^{-s-1} +\f {s(s+1)}2 t^2 (\t+t)^{-s-2} +\cdots +\f { (s)_{k}}{k!} t^k (\t+t)^{-s-k},$$
where $t\ge 0$. A straightforward calculation shows that
$$ g_\t'(t)=-\f{(s)_{k+1}}{k!} t^k (\t+t)^{-s-k-1}\leq 0,\   \ \forall t\ge 0.$$
In particular, this implies that for all $\t, t\in (0, \pi]$,
\begin{equation}\label{e.54}
  0\leq   g_\t (0)-g_\t (t)=\t^{-s}-g_\t(t) \leq \begin{cases} C \t^{-s},\   \     \text{if $0<\t\leq t$};\\
  C t^{k+1} \t^{-s-k-1}, \   \     \text{if $\t> t$}.
  \end{cases}
\end{equation}
It follows that for any $\va>0$,
\begin{align*}
   \sum_{1\leq i<j\leq N}& \rho(z_i, z_j)^{-s} \ge
   \sum_{\ell=0}^k \f {(s)_\ell \va^\ell}{\ell!} \sum_{1\leq i<j\leq N} (\rho(z_i, z_j) +\va)^{-s-\ell}\\
&\ge \f {N^2} 2 \sum_{\ell=0}^k \f {(s)_\ell \va^\ell}{\ell!} \int_{\sph}\int_{\sph}(\rho(x,y) +\va)^{-s-\ell}\, d\s(x)\, d\s(y) -C N\va^{-s},
\end{align*}
where we have used the fact that $F_{-s,\varepsilon} \in \Phi_d$ (by part \eqref{ll2} of Lemma \ref{lem-2-1}) and therefore the energy integral  is smaller than the discrete energy: this can be deduced either from Proposition \ref{prop-1-2} ($\sigma$ minimizes $I_{F_{-s,\va}}$) or from the Stolarsky principle \eqref{e.stol} (since $D^2_{L^2,f_{-s,\va} } (Z_N) \ge 0$).

On the other hand, a direct computation involving \eqref{e.54} shows that
\begin{align*}
    &\Bl|\sum_{\ell=0}^k \f {(s)_\ell \va^\ell}{\ell!} \int_{\sph}\int_{\sph}(\rho(x,y) +\va)^{-s-\ell}\, d\s(x)\, d\s(y)-\int_{\sph}\int_{\sph} \rho(x,y)^{-s}\, d\s(x)\, d\s(y)\Br|\\
   &=c \int_0^\pi (g_\t(0)-g_\t (\va))\sin^{d-1}\t\, d\t\\
  &\leq c \int_0^{\va} \t^{-s+d-1}\, d\t +c\va^{k+1}\int_{\va}^{\pi} \t^{-s-1-k}\sin^{d-1}\t\, d\t\leq c \va^{d-s}. \label{e.compares}
\end{align*}
Thus, putting the above together, we obtain
\begin{align*}
  \sum_{1\leq i<j\leq N} \rho(z_i, z_j)^{-s} &\ge \f {N^2} 2 I_{d,-s}(\sa)-C N\va^{-s} -c N^2\va^{d-s}.
\end{align*}
 Now
 setting $\va=N^{-\f 1d}$, we get the desired  lower estimate:
\begin{equation}\label{}
    E_{d,-s} (Z_N)\ge \f 12 N^2 I_{d, \da}(\sa) -c_d N^{1+\f s d}.
\end{equation}
%The idea of the proof of this part resonates with the proof in of this fact for the classical Riesz energies in \cite{Brauchart}.

\subsection{Proof of  Theorem \ref{thm-1-4-energy-ch14} in the logarithmic case $\delta=0$}

The proof is very similar to the previous case. We shall begin with the upper bound:
\begin{equation}\label{4-4-000}
   \f 2{N^2}  \mathcal{E}_{d,0} ( N) \leq I_{d,0}(\sa)-c_d N^{-1} \log N.
\end{equation}
Recall that  for   $\va\in [0,1)$, by Lemma \ref{lem-2-1},
$F_{0, \va}(t)=\log \big(\frac{\pi}{\arccos t+\va} \big) \in \Phi_d$. Hence, invoking Theorem \ref{thm-3-2}),
 %Then according to  Lemma \ref{lem-2-1},  $F_{-s,\va}\in\Phi_d$.
%Thus, by the Stolarsky principle (part \eqref{St1} of Theorem \ref{thm-3-2}),
\begin{align*}
  \inf_{Z_N}   \frac{2}{N^{2}} &\sum_{1\leq i<j\leq N} \log \bigg(\frac{\pi}{\rho(z_i, z_j)+\va} \bigg)  +N^{-1}  \log \frac{\pi}{\va} \\
     &=\mathcal D^2_{L^2, f_{-s,\va}}(Z) +\int_{\sph}\int_{\sph}  \log \bigg(\frac{\pi}{\rho(x,y)+\va} \bigg)\, d\s(x)\, d\s(y) \\
       &\leq  N^{-1}  \Bl( F_{0,\va}(1)-\min_{|\theta| \leq c_d N^{-\f1d}} F_{0,\va}(\cos \theta)\Br) +  I_{d,0}(\sa)\\
       &= N^{-1}  \log \frac{\pi}{\va}  + N^{-1} \log \bigg( \frac{\pi}{\va + c_d N^{-1/d}} \bigg) +   I_{d,0}(\sa),
    \end{align*}
which yields \eqref{4-4-000} as $\va \rightarrow 0$.\\

To prove the lower bound, we first observe that  $ \bigg| \log \frac{\pi}{\varepsilon + \theta} -  \log \frac{\pi}{\varepsilon + \theta} \bigg| \le \frac12 \varepsilon \theta^{-1}$  % \cdot \min \{\varepsilon^{-1}, \theta^{-1} \}$$
for $\varepsilon$, $\theta >0$.
This easily implies that  $$\displaystyle{\bigg| \int_{\sph} \int_{\sph} \bigg( \log \frac{\pi}{\varepsilon + \rho (x,y) } -  \log \frac{\pi}{\varepsilon + \rho (x,y) } \bigg) d\sigma (x) d\sigma (y) \bigg|  \le C\varepsilon. }$$
Therefore, using the fact that $\sigma$ minimizes $I_{F_{0,\varepsilon}}$,
\begin{align*}
   \sum_{1\leq i<j\leq N}& \log \frac{\pi}{\rho(z_i, z_j)}   \ge
    \sum_{1\leq i<j\leq N} \log \bigg(\frac{\pi}{\rho(z_i, z_j)+\va} \bigg)\\  &
      \ge  \f {N^2} 2  \int_{\sph} \int_{\sph}  \log \frac{\pi}{\varepsilon + \rho (x,y) } \, d\sigma (x) d\sigma (y)  -  \frac12 N \log \frac{\pi}{\varepsilon}   \\
&  \ge  \f {N^2} 2  \int_{\sph} \int_{\sph}  \log \frac{\pi}{  \rho (x,y) } \, d\sigma (x) d\sigma (y)  - C' N^2 \varepsilon - \frac12 N \log \frac{\pi}{\varepsilon} ,
\end{align*}
and the proof is concluded by choosing $\varepsilon = c N^{-1} \log N$.

\section{Proof  of Lemma \ref{lem-4-3}}\label{s.tech}

Recall that
$R_n^\ld(\cos t)=\f{C_n^{\ld}(\cos t)}{C_n^\ld(1)}$ for $\ld>0$, and
$R_n^\ld(\cos t)=\cos nt$ for $\ld=0$. We will use the following formula (\cite[p. 80-81]{Sz}):  \begin{align}\label{5-1}
  (R_n^\ld(t))'
  =\f{n(n+2\ld)}{2\ld+1}R_{n-1}^{\ld+1}(t)=c(n,\ld) R_{n-1}^{\ld+1}(t).  \end{align}

  For the rest of this section, we let $\eta\in C^\infty[0,\infty)$ be such that $\eta(t)=1$ for $0\leq t\leq \f\pi4$ and $\eta(t)=0$ for $t\ge \f \pi2$.

%For $\da>-d$ and $\da\neq 0$, we have
%\begin{align*}
%   \wh{F_\da}(n;\ld)& :=c_{n,\ld} \int_0^\pi \t^\da R_n^\ld(\cos \t) (\sin\t)^{2\ld}\, d\t,
%\end{align*}
%where
%$$ c_{n,\ld}=\f{\Ga(\ld+1)}
%{\Ga(\ld+\f12)\Ga(\f12)}.$$
\subsection{Upper estimate}
The upper estimate is a direct consequence of the following lemma:

\begin{lem}Let $g\in C^\infty [0,\pi]$ and $\ld>0$.
 If $-2\ld-1<\da\leq 1$ and $\da\neq 0$,  then
$$ \Bl|\int_0^{\pi} \ta^\da   g(\ta) R_n^\ld(\cos \t) (\sin \t)^{2\ld}\, d\t\Br|\leq C n^{-2\ld-1-\da}.$$
\end{lem}

\begin{proof}
  We first claim that it suffices to show that
\begin{equation}\label{5-2}
    \Bl|\int_0^{\f \pi 2} \t^\da \eta(\t) g(\t) R_n^\ld(\cos\t) (\sin\t)^{2\ld}\, d\t\Br|\leq C n^{-2\ld-1-\da}.
\end{equation}
Indeed, a slight modification of the proof of  \eqref{5-2} below shows  that for any $g_1\in C^\infty[0,\pi]$,
\begin{align}\label{5-3}
   & \Bl|\int_0^\pi  g_1(\t) R_n^\ld(\cos \t)(\sin\t)^{2\ld}\, d\t\Br|\leq C n^{-2\ld-2}.
\end{align}
Since $\t^\da (1-\eta(\t) ) g(\t)\in C^\infty [0, \pi]$, the desired upper estimates  follow directly from \eqref{5-2} and \eqref{5-3}.

Next, we show \eqref{5-2} in the case when $\ld$ is a positive integer.
Let $\xi_0\in C^\infty(\RR)$ be such that $\xi_0(t)=1$ for $|t|\leq \f 18$ and $\xi_0(t)=0$ for $|t|\ge \f 14$. Let $\xi_1=1-\xi_0$.  Clearly,
\begin{align*}
    \Bl|\int_0^{\pi/2}\t^{\da}\xi_0( \ta n) \eta(\ta) g(\ta) R_n^\ld(\cos \t) (\sin \t)^{2\ld}\, d\t\Br|&\leq C \int_0^{\f 1{4n}} \t^{2\ld+\da}\, d\t\leq C n^{-(2\ld+1+\da)}.
\end{align*}
Thus, it remains  to show that
\begin{align}\label{5-4}
    I_n:=\Bl|\int_0^{\pi/2}\t^{\da}\xi_1( \ta n) \eta(\ta) g(\ta) R_n^\ld(\cos \t) (\sin \t)^{2\ld}\, d\t\Br|&\leq C n^{-2\ld-1-\da}.
\end{align}

To show \eqref{5-4}, set
\begin{equation}\label{5-5}
  D:=\f {d}{d\t}\f 1{\sin\t}=\f 1{\sin\t} \f {d}{d\t}-\f {\cos\t}{\sin^2 \t}.
\end{equation}
Using \eqref{5-1} and    integration by parts $\ld+2$ times, we then obtain
\begin{align*}
    I_n &\leq C n^{-2\ld} \Bl|\int_0^{\f \pi 2} D^\ld\bl(\xi_1( \ta n) \t^{\da}\eta(\ta) g(\ta) \sin^{2\ld}\t\br)   \cos ((n+\ld)\t)\, d\t\Br|\\
    &\leq C n^{-2\ld-2}\Bl|\int_0^{\f \pi 2} \Bl( D^\ld\bl(\xi_1( \ta n) \t^{\da}\eta(\ta) g(\ta) \sin^{2\ld}\t\br)\Br)''   \cos ((n+\ld)\t)\, d\t\Br|.
\end{align*}
Since $\eta'$ is supported in $[\f \pi4, \f \pi2]$, it is easily seen that
\begin{align*}
  \Bl|\Bl(D^\ld\bl(\t^{\da}\xi_1( \ta n) \eta(\ta) g(\ta) \sin^{2\ld}\t\br)\Br)''\Br|\leq C \max_{0\leq j\leq \ld+2} n^{j} |\xi_1^{(j)}(n\t) |\t^{\da+j-2}.
\end{align*}
 Since $\xi_1'(n\t)$ is supported in $[\f 1{8n}, \f 1{4n}]$, it follows that for $1\leq j\leq \ld+2$,
\begin{align*}
  n^{-2\ld+j-2} \int_0^{\f \pi 2} |\xi_1'( \ta n)| \t^{\da+j-2}\, d\t\leq C n^{-2\ld+j-2} n^{-\da-j+1}=C n^{-2\ld-1-\da}.
\end{align*}
On the other hand, since $\xi_1(n\t)$ is supported in $[\f 1{8n}, \infty]$, we have
\begin{align*}
    n^{-2\ld-2}\int_0^{\pi/2} |\xi_1 (\t n)| \t^{\da-2} \, d\t&\leq C n^{-2-2\ld} \int_{\f 1{8n}}^{\pi/2} \t^{\da-2}\, d\t\leq C n^{-2\ld -\da-1}.
\end{align*}
Putting these together, we deduce \eqref{5-4} and hence prove \eqref{5-2} for integer $\ld$.

Finally, we show \eqref{5-2} for all $\ld>0$.
We will use the following formula on ultraspherical polynomials (\cite[(4.10.29), p. 99]{Sz})
\begin{align}\label{5-6}
    (\sin\t)^{2\l} R_n^\ld (\cos \t) =
    \sum_{k=0}^\infty \al_{k,n}^{\ld, \mu} R_{n+2k}^\mu (\cos \t) (\sin \t)^{2\mu},
\end{align}
where $0<\ld<\mu <2\ld+1$ and
 \begin{align}
    \al_{k,n}^{\ld,\mu}&=\f {\Ga(2\ld)\Ga(\mu) 2^{2(\mu-\ld)} (n+2k+\mu)   \Ga(n+k+\mu) \Ga(k+\mu-\ld)}{\Ga(2\mu)\Ga(\mu-\ld) \Ga(\ld)  k! \Ga(n+k+\ld+1)}\notag\\
    &\sim (n+k)^{\mu-\ld} (k+1)^{\mu-\ld-1}.\label{5-7}
 \end{align}
For any $\ld>0$, we can find an integer $\mu$ such that $\ld<\mu<2\ld+1$. Then $-1<-\da<2\ld+1<2\mu+1$. Using \eqref{5-6}, \eqref{5-7} and the already proven case of  \eqref{5-2}, we obtain
\begin{align*}
    &\Bl|\int_0^{\f \pi2} \t^\da \eta(\t) g(\t) R_n^\ld(\cos\t) (\sin\t)^{2\ld}\, d\t\Br|
    =\Bl|\sum_{k=0}^\infty \al_{k,n}^{\ld, \mu} \int_0^{\f \pi2} \t^\da \eta(\ta)R_{n+2k}^\mu (\cos \t) (\sin \t)^{2\mu}\, d\t\Br|\\
    &\leq C \sum_{k=0}^\infty  (n+k)^{\mu-\ld} (k+1)^{\mu-\ld-1} (n+2k)^{-2\mu-1-\da}\leq C n^{-2\ld-1-\da}.
\end{align*}
\end{proof}

\subsection{Lower estimate}

In this subsection, we shall prove the lower estimates: for $0<\da<1$,
\begin{equation}\label{5-8}
    \Bl|\int_0^\pi \t^\da R_n^\ld (\cos\t) (\sin \t)^{2\ld}\, d\t\Br|\ge C n^{-2\ld-1-\da}.
\end{equation}
According to Lemma \ref{lem-2-1}, it suffices to prove \eqref{5-8} as $n\to\infty$.  Next, we assert that if  \eqref{5-8} holds for some $\ld=\mu>0$, then it holds for all $\max\{\f {\mu-1}2,0\}<\ld<\mu$. Indeed, this follows directly from    \eqref{5-7} and Lemma \ref{lem-2-1}:  for $n\ge 1$,
\begin{align*}
     \Bl|\int_0^\pi& \t^\da R_n^\ld(\cos\t) (\sin\t)^{2\ld}\, d\t\Br| =\sum_{k=0}^\infty \al_{k,n}^{\ld, \mu} \Bl|\int_0^\pi \t^\da  R_{n+2k}^\mu (\cos \t) (\sin \t)^{2\mu}\, d\t\Br|\\
    &\ge c \sum_{k=0}^n (n+k)^{\mu-\ld} (k+1)^{\mu-\ld-1} (n+2k)^{-2\mu-1-\da}\ge c n^{-2\ld-1-\da}.
\end{align*}
Thus, according to this last assertion, it suffices to prove \eqref{5-8} in the case when $\ld \ge 2$ is an integer.
Third,  recall that $\eta\in C^\infty [0,\infty)$ is such that $\eta(t)=1$ for $0\leq t\leq \f \pi4$ and $\eta(t)=0$ for $t\ge \f \pi2$, hence, according to  \eqref{5-3},
$$\Bl|\int_0^\pi \t^\da (1-\eta(\t)) R_n^\ld(\cos\t) (\sin \t)^{2\ld}\, d\t\Br|\leq C n^{-2\ld-2}.$$

Putting the above together, we reduce to showing that if  $\ld\ge 2$ is an integer and  $0<\da<1$, then
\begin{equation}\label{5-9}
    \Bl|\int_0^\pi \t^\da \eta(\t) R_n^\ld(\cos\t) (\sin \t)^{2\ld}\, d\t\Br|\ge  C n^{-2\ld-1-\da},\   \  \text{as $ n\to\infty$}.
\end{equation}

For the rest of the proof, we use the notation $c_n$ to denote a positive constant depending on $n$ such that $c_n\sim 1$ as $n\to\infty$.

First, we observe from  \eqref{5-2} for any  $\be\ge 1$, $g\in C^\infty [0,\pi]$ and $\mu\in\NN$,
\begin{equation}\label{5-10}
  \Bl|\int_0^\pi  g(\t) \t^{\be} R_n^{\mu}(\cos\t)(\sin \t)^{2\mu}\, d\t\Br|\leq C_{\mu,g} n^{-2\mu-2}.
\end{equation}
Thus, using \eqref{5-1}, \eqref{5-5}, \eqref{5-10} and integration by parts, we obtain
\begin{align*}
 I_n^\da&:=  \int_0^\pi \t^\da \eta(\t) R_n^\ld(\cos\t) (\sin \t)^{2\ld}\, d\t=c_n n^{-2}\int_0^{\f\pi2} D\Bl(\t^{\da} \eta(\t) \sin^{2\ld}\t\Br) R_{n+1}^{\ld-1} (\cos \t)  \, d\t\\
    &=c_n n^{-2} \int_0^{\f\pi2} \t^\da (\sin\t)^{2\ld-2} \eta(\t)\Bl[ \f{\da \sin \t}{\t} +(2\ld-1)\cos\t\Br]  R_{n+1}^{\ld-1} (\cos \t) \, d\t+O(n^{-2\ld-2})\\
    &=c_n n^{-2}(2\ld-1+\da) \int_0^{\f\pi2} \t^\da (\sin\t)^{2\ld-2} \eta(\t) R_{n+1}^{\ld-1} (\cos \t) \, d\t+O(n^{-2\ld-2}),\end{align*}
    where we recall that  $c_n>0$ and $c_n\sim 1$.
    Continuing in this way  $\ld$ times, we obtain  \begin{align*}
    I_n^\da&=c_n n^{-2\ld} \int_0^{\f \pi2}
    \t^{\da} \eta(\t) \cos ((n+\ld)\t) d\t+O(n^{-2\ld-2}),\end{align*}
    where $c_n'\sim 1$.
   Now  using integration by parts once again, we deduce
    \begin{align*} I_n^\da
    &=-c_n'' n^{-2\ld-1} \int_0^{\pi}
    \t^{\da-1} \eta(\t) \sin ((n+\ld)\t) d\t+O(n^{-2\ld-2}).\end{align*}
    Since
    $\f {\eta(\t)-1}\t=\f{\eta(\t)-\eta(0)}\t\in C^\infty [0,\pi]$, it follows that
    \begin{align*}
      n^{-2\ld-1} \Bl|\int_0^{\pi}
    \t^{\da-1} (\eta(\t)-1) \sin ((n+\ld)\t) d\t\Br|\leq C n^{-2\ld-2}.
    \end{align*}
    This implies that
    \begin{align*}
    -I_n^\da &=c_n'' n^{-2\ld-1} \int_0^{\pi}
    \t^{\da-1}  \sin ((n+\ld)\t) d\t+O(n^{-2\ld-2})\\
    &=c_n''' n^{-2\ld-1-\da} \int_0^{(n+\ld) \pi}
    \t^{\da-1}  \sin \t d\t+O(n^{-2\ld-2})\\
    &\ge c n^{-2\ld-1-\da} \int_0^{2\pi}
    \t^{\da-1}  \sin \t \, d\t+O(n^{-2\ld-2})
    \ge c n^{-2\ld-1-\da}.
\end{align*}
This shows the desired lower estimate.

\end{document}